\documentclass[12pt,leqno]{article}

\usepackage{fullpage}

\usepackage{amsmath,mathrsfs,amsthm,rotating,amsxtra,multirow,graphicx}
\usepackage{txfonts,graphics,float,cancel,extarrows,mathtools}
\usepackage{amsmath, amssymb, amsfonts, amsthm,amscd}
\swapnumbers

\usepackage[all]{xy}

\def\XRow{\rm \,Row\,}
\def\XCol{\rm \,Col\,}

\newtheorem{theorem}{Theorem}[section]
\newtheorem{lemma}[theorem]{Lemma}
\newtheorem{proposition}[theorem]{Proposition}
\newtheorem{corollary}[theorem]{Corollary}

\newtheoremstyle{definition}{3pt}{2pt}{\fontshape{rm}}{}{\bfseries}{.}{ }{}
\theoremstyle{definition}
\newtheorem{definition}[theorem]{Definition}
\newtheorem{remark}[theorem]{Remark}
\newtheorem{conjecture}[theorem]{Conjecture}
\newtheorem{exam*}[theorem]{Example}

\newtheoremstyle{nonum}{}{}{}{}{\bfseries}{.}{ }{#1#3}
\theoremstyle{nonum}

\newtheorem{exams*}{Examples}
\newtheorem{rem*}{Remark}

\newtheoremstyle{newnonum}{}{}{\it}{}{\bfseries}{.}{ }{#1#3}
\theoremstyle{newnonum}
\newtheorem{lemma*}{Lemma}

\newcommand{\beqn}{\begin{equation}}
\newcommand{\eeqn}{\end{equation}}

\DeclareMathOperator{\SL}{SL}
\DeclareMathOperator{\GL}{GL}

\def\bigdot{\bullet}
\DeclareMathOperator{\Perm}{Perm}
\DeclareMathOperator{\End}{End}

\begin{document}

\title{A Study of the representations supported by the orbit closure
  of the determinant}

\author{Shrawan Kumar}

\maketitle

\section{Introduction}

Let  ${\mathfrak v}$ be a complex vector space of dimension $m$  and let $E:=
\End {\mathfrak v}$.  Consider $\mathscr{D}\in Q:=S^m(E)^*$, where $\mathscr{D}$ is the function taking
determinant of any $X\in \End {\mathfrak v}$.
Fix a basis
 $\{ v_1, \dots, v_m \}$ of ${\mathfrak v}$ and a positive integer $n<m$ and
 consider the
 function $\mathscr{P}\in Q$, defined by $\mathscr{P}(X) =
x^{m-n}_{1,1} \Perm (X^o),$ $X^o$ being the component of $X$ in the right
 down $n\times n$
corner, where any element of $\End {\mathfrak v}$ is represented by
a $m\times m$-matrix $X=(x_{i,j})_{1\leq i,j,\leq m}$ in the basis
$\{v_i\}$ and Perm denotes the permanent. The group $G=GL(E)$
canonically acts on $Q$. Let $\mathcal{X}$ (resp. $\mathcal{Y}$) be
the $G$-orbit closure of $\mathscr{D}$ (resp. $\mathscr{P}$) inside
$Q$. Then, $\mathcal{X}$ and $\mathcal{Y}$ are closed (affine)
subvarieties of $Q$ which are stable under the standard homothety
action of $\mathbb{C}^*$ on $Q$. Thus, their affine coordinate rings
$\mathbb{C}[\mathcal{X}]$ and $\mathbb{C}[\mathcal{Y}]$ are
nonnegatively graded $G$-algebras over the complex numbers
$\mathbb{C}$. Clearly, $\mathscr{D}\odot \End E\subset \mathcal{X}$,
where $\End E$ acts on $Q$ on the right via: $(q\odot A)(X)=q(A\cdot
X)$, for $A\in \End E, q\in Q$ and $X\in E$.

For any positive integer $n$, let $\bar{m}=\bar{m}(n)$ be the smallest
positive integer such that the permanent of any $n\times n$ matrix can be
realized as a linear projection of the determinant of a $\bar{m}\times \bar{m}$
matrix. This is equivalent to saying that $\mathscr{P} \in \mathscr{D}\odot \End E
$ for the pair
$(\bar{m},n)$. Then,  Valiant conjectured that the function $\bar{m}(n)$ grows faster than
any polynomial in $n$ (cf. \cite{V}).

Similarly, let $m=m(n)$ be the smallest integer such that $\mathscr{P}\in \mathcal{X}$
(for the pair $(m,n)$). Clearly, $m(n)\leq \bar{m}(n)$. Now,
Mulmuley-Sohoni strengthened Valiant's conjecture. They conjectured that,
in fact, the function $m(n)$ grows faster than
any polynomial in $n$ (cf.\cite {MS1}, \cite{MS2} and the references therein). They further
conjectured that
if $\mathscr{P}\notin \mathcal{X}$, then there exists an irreducible $G$-module
which occurs in
$\mathbb{C}[\mathcal{Y}]$ but does not occur in $\mathbb{C}[\mathcal{X}]$. (Of
course, if $\mathscr{P}\in \mathcal{X}$,
then $\mathbb{C}[\mathcal{Y}]$ is a $G$-module quotient of
$\mathbb{C}[\mathcal{X}]$.) This Geometric Complexity Theory
programme initiated by Mulmuley-Sohoni provides a significant mathematical approach to
solving the Valiant's conjecture (in fact, strengthened version of
Valiant's conjecture proposed by them).

By \cite[Theorem 5.2]{K}, if an irreducible $G$-module $V_{E}(\lambda)$
(with highest weight $\lambda$) appears in $\mathbb{C}[\mathcal{Y}]$,
then $V_{E}(\lambda)$ is a polynomial representation of $G$ given by a
partition
$$
\lambda:\left(\lambda_{1}\geq \lambda_{2}\geq \ldots\geq
\lambda_{n^{2}+1}\geq 0\geq \ldots \geq 0\right)
$$
with last $m^{2}-(n^{2}+1)$ zeroes.

From now on (in this Introduction), we assume that $m$ is even. Our principal
result in this paper (Corollary \ref{coro3.2}) asserts
that for  any partition $\mu : (\mu_1\geq \ldots \geq
\mu_{m}\geq 0\geq \ldots \geq 0)$ with last $m^{2}-m$ zeroes, the
irreducible $G$-module $V_{E}(m\mu)$ appears in
$\mathbb{C}[\mathcal{X}]$ with nonzero multiplicity, provided the column Latin
$(m,m)$-square conjecture is valid (cf. Conjecture \ref{conj16}). In particular,
if $m\geq n^{2}+1$, for any irreducible representation $V_{E}(\lambda)$
appearing in $\mathbb{C}[\mathcal{Y}]$, $V_{E}(m\lambda)$ appears in
$\mathbb{C}[\mathcal{X}]$ (again asuming the validity of the column Latin
$(m,m)$-square conjecture). Thus, finding an irreducible representation
in $\mathbb{C}[\mathcal{Y}]$ which does not occur in
$\mathbb{C}[\mathcal{X}]$ (on which the success of the  Mulmuley-Sohoni
programme relies)
 for
$m\geq n^{2}+1$ is not so easy. As a consequence of our Corollary \ref{coro3.2},
we deduce that
the symmetric Kronecker coefficient
$sk_{m\overline{\lambda},d{\delta}_{m},d{\delta}_{m}}>0
$ for any partition $\overline{\lambda}:\left(\overline{\lambda}_{1}\geq
\overline{\lambda}_{2}\geq \cdots\geq
\overline{\lambda}_{m}\geq 0\right)$ of $d$, where
${\delta}_{m}$ is the partition
${\delta}_{m}:(1\geq 1\geq \cdots \geq 1)$ ($m$ factors)
(cf. Corollary \ref{coro3.5}).

Here is the content of this paper:

Section 2: By a result of
Howe (cf. Corollary \ref{coro2.4}), for any fundamental weight $\delta_{i}$ $(1\leq i\leq
m^{2}=\dim E)$ of $GL(E)$, the irreducible $GL(E)$-module
$V_{E}(d\delta_{i})$, for $0<d<m$, does not occur in
$S^{\bigdot}(S^{m}(E))$, whereas $V_{E}(m\delta_{i})$ occurs with
multiplicity one in $S^{\bigdot}(S^{m}(E))$. In fact, it occurs in
$S^{i}(S^{m}(E))$.
We give an explicit construction of the highest weight vector
$P_{i}=\gamma_{m,i}$ in this unique copy of $V_{E}(m\delta_{i})$ in
$S^{i}(S^{m}(E))$ (cf. Definition \ref{def1}).

Section 3: For any $1 \leq i \leq m$, we calculate $\gamma_{m,i}$ on a certain subset $\theta(M(m,i))$ of
$\mathcal{X}$ given by a morphism $\theta: M(m,i) \to \mathcal{X}$, where
$M(m,i)$ denotes the set of $m\times i$-matrices. The induced map $\theta^*$
on the level of affine coordinate rings is identified with a certain very explicit map $\varphi$.
The main result of this section is Proposition \ref{prop8}, which asserts that
$\gamma_{m,i}$ restricted to the image $\theta (M(m,i))$ is nonzero if and only if
the $\GL(V_{m})$-submodule $U_i$ generated by $v_{o}^{\otimes i}\in
S^{i}(S^{m}(V^{*}_{m}))$ intersects
the  isotypic
component $\mathcal{I}_{m\delta_i}$ of
$S^i(S^{m}(V^{*}_{m}))$ corresponding to the irreducible $\GL(V_{m})$-module
$V_m(m\delta_i)^*$ nontrivially,
where the element $v_o$ is defined by the identity \eqref{e6}.

In Section 4, we recall the Latin squares (more generally Latin rectangles) and state
the celebrated {\it Latin square conjecture} due to Alon-Tarsi and an equivalent
formulation due to Huang-Rota called the {\it column Latin square conjecture}.
We recall that the Latin square conjecture is known to be true for $p-1$ as well
as $p+1$ for any odd prime $p$; in particular, it is true for any even integer
up to $24$ (cf. Remark \ref{rem18}).

Section 5 is devoted to proving that the validity of the column Latin square
conjecture implies that the isotypic component $\mathcal{I}_{m\delta_i}$ of
$S^i(S^{m}(V^{*}_{m}))$ corresponding to the irreducible $\GL(V_{m})$-module
$V_m(m\delta_i)^*$ intersects the $\GL(V_{m})$-submodule $U_i$ generated by
$v_{o}^{\otimes i}$ nontrivially (cf. Theorem \ref{thm23}). In fact, for $i=m$,
 we show that the latter assertion is equivalent to the column Latin square
conjecture.

This sets the stage for the proof of our main theorem  (cf. Theorem \ref{thm3.1}),
which asserts that the irreducible module $V_E(m\delta_i)$ occurs in
$\mathbb{C}[\mathcal{X}]$ with multiplicity one for any $1\leq i\leq m$ if the
column Latin square conjecture is true for $m\times m$ Latin squares. This is shown
by proving that $P_{i}$
does not vanish identically on  $\mathcal{X}$. As an immediate
 corollary (cf. Corollary \ref{coro3.2}), we deduce that for any partition $\lambda: \lambda_1 \geq \dots \geq
 \lambda_m\geq 0, V_E(m\lambda)$ occurs in $\mathbb{C}[\mathcal{X}]$ (if the
column Latin square conjecture is true for $m\times m$ Latin squares).

Finally, in Remark \ref{rem5.4} (b), we observe that $V_{E}(m\delta_{i})$
(for any $1\leq i\leq m$)
occurs in $\mathbb{C}[\overline{GL(E)\cdot \mathfrak{P}}]$ with multiplicity one,
where $\mathfrak{P}$ is the function $E\to \mathbb{C}$
  taking any matrix $A\in E=\End \mathfrak{v}$ to its permanent. (Of course, as
  mentioned above, $V_E(d\delta_i)$, for any $0<d<m$ and $1\leq i\leq m^2$,
  does not occur in
$S^{\bigdot}(S^{m}(E))$, and hence  it does not occur in
$\mathbb{C}[\overline{GL(E)\cdot \mathfrak{P}}]$ or in $\mathbb{C}[\mathcal{X}]$.

{\it For any vector space $W$ over the complex numbers, in this paper, we view
$S^k(W)$ as the subspace of $\otimes^k W$ consisting of symmetric tensors.}

  \vspace{6mm}
  \noindent
  {\bf Acknowledgements:}  I am grateful to P. B\"urgisser and C. Ikenmeyer for pointing
  out a gap in an earlier version of the paper which resulted in this vastly
  modified version. The gap did not alter the main theorem \ref{thm3.1}, however
  unfortunately now the validity of the theorem is proved only under the hypothesis
  that the Latin Square Conjecture is valid.

  I thank J. Landsberg for many helpful
  conversations/correspondences; in particular, Section 5 is greatly influenced
  by discussions with him and Proposition \ref{lem13} (b) is due to him.

  Partial support from the NSF grants is gratefully acknowledged.

\section{An explicit realization of multiples of fundamental
  {\boldmath$GL(E)$}-representations in
  {\boldmath$S^{\bigdot}(S^{\bigdot}(E))$}}\label{sec2}

Let $E$ be a finite dimensional complex vector space with basis
$\{e_{1},\ldots,e_{\ell}\}$. Let $\delta_{i}$, $1\leq i\leq \ell$, be
the $i$-th fundamental weight of $GL(E)=GL(\ell)$. This corresponds to the partition
$ 1 \geq 1\geq \dots \geq 1$ ($i$-factors).

\begin{lemma}\label{lem2.1}
For any positive integers $d$, $j$ and $m$, the multiplicity of the
irreducible $GL(E)$-module $V_{E}(d\delta_{i})$ (with highest weight
$d\delta_{i}$) in $S^{j}(S^{m}(E))$ is the same as the multiplicity of
the irreducible $GL(E_{i})$-module $V_{E_{i}}(d\delta_{i})$ in
$S^{j}(S^{m}(E_{i}))$, where $E_{i}$ is the subspace of $E$ spanned by
$\{e_{1},\ldots,e_{i}\}$.

In fact, the highest weight vectors in
$S^{j}(S^{m}(E))$ for the irreducible
$GL(E)$-module $V_{E}(d\delta_{i})$ coincide with the highest weight
vectors in $S^{j}(S^{m}(E_{i}))$ for the irreducible $GL(E_{i})$-module
$V_{E_{i}}(d\delta_{i})$.
\end{lemma}

\begin{proof}
Let $B_{E}$ be the standard Borel subgroup of $GL(E)$ consisting
of all the invertible upper triangular matrices (with respect to
the basis $\{e_{1},\ldots,e_{\ell}\}$). Let $v\in S^{j}(S^{m}(E))$ be a
$B_{E}$-eigenvector of weight $d\delta_{i}$. Then, clearly $v\in
S^{j}(S^{m}(E_{i}))$ and $v$ is a $B_{E_{i}}$-eigenvector of weight
$d\delta_{i}$. Conversely, let $v'\in S^{j}(S^{m}(E_{i}))$ be a
$B_{E_{i}}$-eigenvector of weight $d\delta_{i}$. Then, the line
$\mathbb{C}v'$ is clearly stable under $B_{E}$. Moreover, the vector
$v'$ is a weight vector of weight $d\delta_{i}$ with respect to the
standard maximal torus $T_{E}$ (consisting of invertible diagonal
matrices) of $GL(E)$. This proves the lemma.
\end{proof}

\begin{corollary}\label{coro2.2}
With the notation as above, the multiplicity $\mu_{E}(d\delta_{i})$ of
$V_{E}(d\delta_{i})$ in $S^{j}(S^{m}(E))$ is equal to the dimension of
the invariant space $[S^{j}(S^{m}(E_{i}))]^{SL(E_{i})}$ if $di =
jm$. If $di\neq jm$, $\mu_{E}(d\delta_{i})=0$.
\end{corollary}

We recall the following result from \cite[Proposition 4.3]{H}.

\begin{proposition}\label{prop2.3} Let $E$ be a vector space of dimension $\ell$ as above.
For positive integers $j$, $m$, we have
\begin{itemize}
\item[\rm(a)] $[S^{j}(S^{m}(E))]^{SL(E)}=(0)$, if $0<j<\ell$

\item[\rm(b)] $[S^{\ell}(S^{m}(E))]^{SL(E)}\simeq
\begin{cases}
(0), & \text{if $m$ is odd}\\
\mathbb{C}, & \text{if $m$ is even.}
\end{cases}$
\end{itemize}
\end{proposition}

Combining Corollary \ref{coro2.2} with Proposition \ref{prop2.3},
together with the action of the center of $GL(E)$, we get the
following result.

\begin{corollary}\label{coro2.4} Let $E$ be a vector space of dimension $\ell$ as above.
Let $m$ be a positive even integer and let $l\leq i\leq \ell$. Let
$d$ be the smallest positive integer such that $V_{E}(d\delta_{i})$
occurs in $S^{\bigdot}(S^{m}(E))$ as a $GL(E)$-submodule. Then,
$d=m$. Moreover, $V_{E}(m\delta_{i})$ occurs in
$S^{\bigdot}(S^{m}(E))$ with multiplicity 1 and it occurs
in $S^{i}(S^{m}(E))$.
\end{corollary}

{\it From now on,  $m$ is an even positive
integer.}

\vskip1ex

We first give an explicit construction of the invariant
$[S^{i}(S^{m}(E_{i}))]^{SL(E_{i})}$ for any $1\leq i\leq \ell$. Recall
from Proposition \ref{prop2.3} that it is one dimensional.

\begin{definition} \label{def1}
{\bf (An explicit construction of}
  $[S^{i}(S^{m}(E_{i}))]^{SL(E_{i})}${\bf )}

Recall that $E_{i}$ has a basis $\{e_{1},\ldots,e_{i}\}$. Let $M(i,i)$
be the space of $i\times i$ matrices over $\mathbb{C}$. Define a
linear isomorphism
\[
\theta : \left(\otimes^{2}E_{i}\right)^{*}\xrightarrow{\sim} M(i,i),\,\,\,\,
\theta (f) = \left(\theta(f)_{p,q}\right)_{1\leq p, q\leq i},\]
where
$\theta(f)_{p,q}=f\left(e_{p}\otimes e_{q}\right),$  for any $f\in
  \left(\otimes^{2}E_{i}\right)^{*}.$

Let $GL(E_{i})$ act on $M(i,i)$ via
$$
g\cdot A=\left(g^{-1}\right)^{t}Ag^{-1},\,\,\,
\text{for}\,\, g\in GL(E_{i})\,\, \text{and} \,\,A\in M(i,i).$$

Then, $\theta$ is $GL(E_{i})$-equivariant.
Now, define the map (setting $m'=m/2$)
$$
\theta^{\otimes m'}:
\left(\otimes^{m}E_{i}\right)^{*}\xrightarrow{\sim}\otimes^{m'}(M(i,i))
$$
by identifying
$$
\left(\otimes^{m}E_{i}\right)^{*}\simeq
\left(\left(\otimes^{2}E_{i}\right)^{*}\right)
\displaystyle{\otimes\cdots\otimes}\left(\left(\otimes^{2}E_{i}\right)^{*}\right) \quad {(m'
  \ \text{factors})}
$$
and setting
$$
\theta^{\otimes m'}\left(f_{1}\otimes \cdots\otimes
f_{m'}\right)=\theta(f_{1})\otimes\cdots\otimes \theta(f_{m'}), \,\,\,
\text{for}\,\,  f_{k}\in (\otimes^{2}E_{i})^{*}.$$

For any finite dimensional vector space $W$ and nonnegative integer $k$,
let $\mathcal{P}^{k}(W)$ be the space of homogeneous polynomials of
degree $k$ on $W$, i.e.,
$$
\mathcal{P}^{k}(W)=\{f:W\to \mathbb{C}\text{~ such that~ }f(\lambda
w)=\lambda^{k}f(w) \ \forall \ w\in W\text{~ and~ }\lambda\in
\mathbb{C}\}.
$$

Then, there is a linear isomorphism (cf. [GW, Proposition B.2.4]).
$$
\xi : S^{k}(W)^{*}\xrightarrow{\sim}\mathcal{P}^{k}(W)
$$
defined by $\xi(\theta)(w)=\theta(w^{\otimes k})$, for $\theta\in
S^{k}(W)^*$ and $w\in W$.
If an algebraic group $G$ acts linearly on $W$, then $\xi$ is
$G$-equivariant.

Define the linear map $\bar{\xi}: \mathcal{P}^{k}(W)\to (\otimes^{k}W)^{*}$
by
$$
(\bar{\xi}(f))(w_{1}\otimes\dots \otimes w_{k})=\frac{1}{k!}~ (\text{the coefficient
  of } t_{1}\dots t_{k}\text{~ in~ } f(t_{1}w_{1}+\cdots+t_{k}w_{k})),
$$
for $f\in \mathcal{P}^{k}(W)$ and $w_{1},\dots,w_{k}\in W$. Then,
the inverse map
$$
\xi^{-1}:\mathcal{P}^{k}(W)\to S^{k}(W)^{*}
$$
is given by the composition of $\bar{\xi}$ with the restriction map
$(\otimes^{k}W)^{*} \to S^{k}(W)^{*}.$

Consider the linear projection obtained via the symmetrization
$$\pi: \otimes^m E_i\to
S^m(E_i),\,\,\, w_{1}\otimes\dots \otimes w_{m} \, \mapsto \frac{1}{m!}\sum_{\sigma
\in \mathfrak{S}_m}\, w_{\sigma (1)}\otimes\dots \otimes w_{\sigma (m)},$$
where $\mathfrak{S}_m$ is the permutation group on the symbols $[m]:=
\{1, \dots, m\}.$ Thus, we have
$\GL(E_{i})$-equivariant linear maps
$$
S^{m}(E_{i})^{*}\xhookrightarrow{\pi^*}
(\otimes^{m}E_{i})^*\mathop{\xlongrightarrow{\sim}}^{\theta^{\otimes
    m'}} \otimes^{m'}(M(i,i)).
$$

This gives rise to a linear $\GL(E_{i})$-equivariant map
\begin{equation}\label{e1}
S^{i}(S^{m}(E_{i})^*)\to S^{i}(\otimes^{m'}(M(i,i))).
\end{equation}

Now, consider the map
\[
\det\nolimits^{\otimes m'}:  M(i,i)\to \mathbb{C},\,\,\,
  A\mapsto (\det A)^{m'}.
\]
It is clearly a homogeneous polynomial of degree $im'$,
which is $\SL(E_{i})$-invariant. Thus, via the above isomorphism
$\xi^{-1}$, we get a $\SL(E_{i})$-invariant linear map
\begin{equation} \label{e2}
\widehat{\det\nolimits^{\otimes m'}}:S^{i\, m'}(M(i,i))\to
\mathbb{C}.
\end{equation}

Of course, we have a canonical $\GL(E_{i})$-equivariant projection
\begin{equation}
S^{i}(\otimes^{m'}(M(i,i)))\to S^{i\, m'}(M(i,i)),
\end{equation}
obtained via the inclusion
$$S^{i}(\otimes^{m'}(M(i,i))) \subset \otimes^{i}(\otimes^{m'}(M(i,i)))
\simeq \otimes^{im'}(M(i,i))$$
followed by the symmetrization: $\otimes^{im'}(M(i,i)) \to S^{im'}(M(i,i)).$

Composing the maps $(2)\circ (3)\circ (1)$, we get a $\SL(E_{i})$-invariant
linear map
$$
\gamma_{m,i}:S^{i}(S^{m}(E_{i})^{*})\to \mathbb{C}.
$$
For any vector space $W$, we have a canonical $\GL(W)$-equivariant identification
\begin{equation}\label{e4} S^i(W^*)\simeq S^i(W)^*
\end{equation}
via $S^i(W^*) \subset \otimes^i(W^*)\simeq (\otimes^i W)^*\to  S^i(W)^*$, where
the last map is the restriction map.
Thus, $\gamma_{m,i}$ can be thought of as an element of
$[S^{i}(S^{m}(E_{i}))]^{\SL(E_{i})}$.

\end{definition}

\begin{lemma}
$$
\gamma_{m,i}\left(\left(\sum^{i}_{j=1}(e^{*}_{j})^{\otimes m}\right)^{\otimes
  i}\right)\neq 0,
$$
where $\{e^{*}_{1},\dots, e^{*}_{i}\}$ is the basis of $E^{*}_{i}$
dual to the basis $\{e_{1},\dots,e_{i}\}$ of $E_{i}$.
\end{lemma}

\begin{proof} Let $E_{j,j}\in M(i,i)$ be the matrix with  all the
entries $0$, except
$(j,j)$ which is $1$.
By the definition,
\begin{align*}
\gamma_{m,i}\left(\left(\sum^{i}_{j=1}(e^{*}_{j})^{\otimes m}\right)^{\otimes
  i}\right)
&=\gamma_{m,i}\left(\sum_{1\leq j_{1},\dots,j_{i}\leq
  i}\left((e^{*}_{j_{1}})^{\otimes m}\otimes\cdots\otimes
(e^{*}_{j_{i}})^{\otimes m}\right)\right)\\
&= \sum_{1\leq j_{1},\dots,j_{i}\leq
  i}\widehat{\det\nolimits^{\otimes
    m'}}\left(E_{j_{1},j_{1}}^{\otimes m'}\otimes\cdots\otimes
E_{j_{i},j_{i}}^{\otimes m'}\right)\\
& =\sum\limits_{1\leq j_{1},\dots,j_{i}\leq i}\frac{1}{(i\,
  m')!}\quad\text{the coefficient of}\quad (t_{1}\dots t_{i\, m'})\,\,\text{in}\\
& [\det
    ((t_{1}+\cdots+t_{m'})E_{j_{1},j_{1}}+(t_{m'+1}+\cdots+t_{2\,
      m'})E_{j_{2},j_{2}}+\cdots+ \\
    &   (t_{(i-1)m'+1}+\cdots+t_{i\,
      m'})E_{j_{i},j_{i}})]^{m'}\\
&= \sum_{\sigma\in \mathfrak{S}_{i}}\frac{1}{(i\, m')!}\text{~~ the
    coefficient of~~ } (t_{1}t_{2}\dots t_{i\, m'})\,\,\text{in}\\
& \left[(t_{1}+\cdots+t_{m'})^{m'}\dots
    (t_{(i-1)m'+1}+\cdots+t_{i\, m'})^{m'}\right]\\
&= \frac{i!}{(i\, m')!}\left(m'!\right)^{i}\neq 0.
\end{align*}
This proves the lemma.
\end{proof}

We record this in the following.

\begin{lemma}\label{lem2.6}
The element $\gamma_{m,i}$ is the unique (up to a scalar multiple) nonzero
 element of  $\,\,[S^{i}(S^{m}(E_{i}))]^{SL(E_{i})}$.
\end{lemma}

\section{Calculation of $\gamma_{m,i}$ on the determinant orbit closure}
We continue to assume that $m$ is even and $m'=m/2$.

Let $\mathfrak{v}$ be a complex vector space of dimension $m$ and let
$E:=\End \mathfrak{v}=\mathfrak{v}\otimes \mathfrak{v}^{*}$.
Fix a basis $\{v_{1}, \ldots, v_{m}\}$ of $\mathfrak{v}$ and let
$\{v^{*}_{1}, \ldots, v^{*}_{m}\}$ be the dual basis of $\mathfrak{v}^{*}$. Take
the basis $\{v_{i}\otimes v^{*}_{j}\}_{1\leq i,j\leq m}$ of $E$ and
order them as $\{e_{1},e_{2},\ldots,e_{m^{2}}\}$ satisfying
$$
e_{1}=v_{1}\otimes v^{*}_{1},\quad e_{2}=v_{2}\otimes
v^{*}_{2},\ldots,e_{m}=v_{m}\otimes v^{*}_{m}.
$$
Fix $1\leq i\leq m$ and consider the subspace $E_i$ of $E$ spanned by
$\{e_{1},\dots,e_{i}\}$.  Take $A\in \End\,E$ of
the form
$$
Ae_{j}=\sum^{m}_{p=1}a^{j}_{p}e_{p}, \ 1\leq j\leq i.
$$

In the sequel, we will only consider $A\in \End\, E$ of
the above form and the values of $Ae_{j}$ for $j>i$ will be irrelevant for us.
Thus, we can (and will) think of $A=(a^j_p)_{1\leq p\leq m, 1\leq j\leq i}$ as an
$m\times i$-matrix.

Define a right action of the semigroup $\End\, E$ on
$Q:=\mathcal{P}^m(E)\simeq S^{m}(E)^{*}$  (cf. Definition \ref{def1} for the
last identification under $\xi$) via
\beqn\label{e-1}
(f\odot A) (e) = f(Ae),\quad\text{for}\,\,\, f\in Q, \ A\in \End
  \,E\,\,\,\text{and}\,\,\, e\in E.
\eeqn

Take $f=\mathscr{D}\odot A\in Q$, where $\mathscr{D}\in \mathcal{P}^m(E)$ is the
function taking the determinant of any $X\in E$. For $1\leq l_1, \dots, l_m \leq i$,
let $A^{l_{1}, \dots , l_m}$
denote the $m\times m$
matrix with the first column $\left[\begin{smallmatrix}
    a^{l_{1}}_{1}\\ \vdots\\ a^{l_{1}}_{m}\end{smallmatrix}\right]$ etc.
Then, the  image of $f_{|E_i}$ in $\otimes^{m'}M(i,i)$ under $\theta^{\otimes\,
  m'}\circ \pi^{*}$ (cf. Definition \ref{def1}) is given by
\begin{align*}
& \sum_{1\leq j_{p},k_{p}\leq i}f\left((e_{j_{1}}\otimes
  e_{k_{1}})\otimes\cdots\otimes (e_{j_{m'}}\otimes
  e_{k_{m'}})\right)E_{j_{1},k_{1}}\otimes\cdots\otimes
  E_{j_{m'},k_{m'}}\\
&= \frac{1}{m!}\sum_{1\leq j_{p},k_{p}\leq
    i}\Perm\left(A^{j_{1},k_{1}, \dots , j_{m'},k_{m'}}\right)E_{j_{1},k_{1}}\otimes\cdots\otimes
  E_{j_{m'},k_{m'}}\\
&=\frac{1}{m!}\sum_{\substack{d_{1}+\cdots +d_{i}=m\\ d_{j}\geq 0}}\Perm
A^{(d_{1},\dots,d_{i})}\sum_{1\leq j_{p},k_{p}\leq
  i}E_{j_{1},k_{1}}\otimes\cdots\otimes E_{j_{m'},k_{m'}},
\end{align*}
where the last summation runs over those ordered $m$-tuples
$(j_{1},k_{1}, \dots , j_{m'},k_{m'})$ such that the collection
(without regard to the order)
$$
\left\{j_{1},k_{1},\dots,j_{m'},k_{m'}\right\}=\left\{1^{d_{1}},2^{d_{2}},\dots,
i^{d_{i}}\right\},
$$
 $\{1^{d_{1}},2^{d_{2}},\dots,i^{d_{i}}\}$ means the collection
$$
\left\{\underbracket{1,\dots,1}_{d_{1}\text{-times}};
\underbracket{2,\dots, 2}_{d_{2}\text{-times}};\dots;
\underbracket{i,\dots,i}_{d_{i}\text{-times}}\right\},
$$
$A^{(d_{1},\dots,d_{i})}$ means the $m\times m$ matrix with
columns
$$
\underbracket{A^{1},\dots,A^{1}}_{d_{1}\text{-times}}; \ \underbracket{A^{2},
\dots,A^{2}}_{d_{2}\text{-times}},\dots,\underbracket{A^{i},\dots,A^{i}}_{d_{i}
\text{-times}}\,,\quad\,\,
A^{j}\,\,\,\text{is the column} \,\,\left[\begin{smallmatrix}
    a^{j}_{1}\\ \vdots\\ a^{j}_{m}\end{smallmatrix}\right],$$
    and Perm denotes the permanent  of the matrix.

On the vector space $M(m,i)$ of $m\times i$-matrices, $\GL(m)\times \GL(i)$ acts via:
$$(g,h)\cdot X= gXh^{-1}, \,\,\,\text{for}\,\,g\in \GL(m), h\in  \GL(i), X\in
M(m,i).$$
In particular, the permutation group $\mathfrak{S}_m$ thought of as the subgroup of
permutation matrices in $\GL(m)$ acts on $M(m,i)$ and hence on any
$\mathcal{P}^k(M(m,i)).
$
For any $\mathfrak{d}=(d_{1},\dots,d_{i})$, $d_{1}+\cdots+d_{i}=m$ and
$d_j\geq 0$, set
$$a_{\mathfrak{d}}(A)= \Perm A^{(d_{1},\dots,d_{i})}.$$
Then, clearly
$$
a_{\mathfrak{d}}\in
\mathcal{P}^{m}(M(m,i))^{-(\epsilon_{1}+\cdots+\epsilon_{m}), \mathfrak{S}_{m}},
$$
where the superscript `$-(\epsilon_{1}+\cdots+\epsilon_{m}),\mathfrak{S}_{m}$'
denotes the $\mathfrak{S}_{m}$-invariants of weight
$-(\epsilon_{1}+\cdots+\epsilon_{m})$ with respect to the action of
$\GL(m)$, i.e., the invertible diagonal matrices $(t_1, \dots, t_m)$ act via
$(t_1 \dots t_m)^{-1}$. Recall from [GW, Theorem 5.6.7] that, as
$\GL(m)\times \GL(i)$-modules, for any $j\geq 0$,
\begin{equation}\label{e5}
\mathcal{P}^{j}(M(m,i))\simeq \sum_{\substack{\mu:\mu_{1}\geq
    \mu_{2}\geq \dots\geq \mu_{i}\geq
    0\\ |\mu|=j}}V_{m}(\mu)^*\otimes V_{i}(\mu),
\end{equation}
where $|\mu|:=\sum \mu_i$ and $V_{m}(\mu)$ denotes the irreducible $\GL(m)$-module
corresponding to the partition $\mu$.

Let $V_m:=\mathbb{C}^m$ with the standard basis $\{v_1, \dots, v_m\}$.
Define the elements
\beqn\label{e6}
\overline{v}_{o} := v^{*}_{1}\otimes\cdots\otimes
v^{*}_{m} \in \otimes^{m}(V^*_{m});\,\,\, v_{o} :=
\frac{1}{m!}\sum_{\sigma \in
  \mathfrak{S}_{m}}\sigma\cdot \overline{v}_{o}\in S^{m}(V^*_{m}),
  \eeqn
and
\beqn\label{e100}\mathfrak{v}_o:=\frac{1}{m!}\sum_{\sigma \in
  \mathfrak{S}_{m}}\,\sigma\cdot (v_1\otimes \dots \otimes v_m)\in S^m(V_m).
  \eeqn
From \eqref{e5}, by a classical result due to Kostant [Ko]  (which asserts that
for any irreducible $\SL(V_m)$-module $V_m(\lambda)$ corresponding to the partition
$\lambda$ with $|\lambda|=m$, its zero weight space is an irreducible
representation $W_\lambda$ of $\mathfrak{S}_m$ corresponding to the partition
$\lambda$), we get
\begin{align*}
&
  \mathcal{P}^{m}(M(m,i))^{-(\epsilon_{1}+\cdots+\epsilon_{m}),\mathfrak{S}_{m}}
  \simeq \left(S^{m}(V_{m})^{\epsilon_{1}+\cdots+\epsilon_{m},\mathfrak{S}_{m}}
  \right)^{*}\otimes
  S^{m}(V_{i}), \,\,\,\text{and}\\
& S^{m}(V_{m})^{\epsilon_{1}+\cdots+\epsilon_{m},\mathfrak{S}_{m}}\simeq
  \mathbb{C} \mathfrak{v}_o.
\end{align*}
Thus,
\begin{equation}\label{e7}
\mathcal{P}^{m}(M(m,i))^{-(\epsilon_{1}+\cdots+\epsilon_{m}),\mathfrak{S}_{m}}\simeq
S^{m}(V_{i}),
\end{equation}
as $\GL(i)$-modules. It is easy to see that $\{a_{\mathfrak{d}}\}_{\mathfrak{d}=(d_{1},\dots,d_{i})}$ with
$|\mathfrak{d}|=m$ are linearly independent (by taking , e.g., $a^j_p=a_1^j$ for all
$1\leq p \leq m$).
Hence, $\{a_{\mathfrak{d}}\}_{|\mathfrak{d}|=m}$
provides a basis of
$S^{m}(V_{i})$ under the above identification \eqref{e7}. The $\GL(i)$-module
 isomorphism
\eqref{e7} induces a $\GL(i)$-algebra homomorphism:
\[
\varphi : S^{\bullet}(S^{m}(V_{i}))
 \to
\mathcal{P}^{m\bullet}(M(m,i))^{-\bullet(\epsilon_{1}+\cdots+\epsilon_{m}),
 \mathfrak{S}_m}\simeq \oplus \left(V_{m}(\mu)^{\bullet(\epsilon_{1}+\cdots+
 \epsilon_{m}),
 \mathfrak{S}_m}\right)^*\otimes V_{i}(\mu),
\]
where the above sum  runs over $\mu:\mu_{1}\geq \dots\geq \mu_{i}\geq 0$,
$|\mu|=m\bullet$; the last isomorphism follows by the identity \eqref{e5}.

We now give an alternative description of the map
$$
\varphi: S^{\bullet}(S^{m}(V_{i}))\to
\mathcal{P}^{m\bullet}(M(m,i))^{-\bullet(\epsilon_{1}+\cdots+\epsilon_{m}), \mathfrak{S}_m}.
$$
First of all, as $\GL(m) \times \GL(i)$-modules,
\beqn\label{e8}
\mathcal{P}^{mj}(M(m,i)) \simeq \mathcal{P}^{mj}(V_{m}\otimes
  V^{*}_{i})
\simeq S^{mj}(V^{*}_{m}\otimes V_{i}),
\eeqn
where the last identification is obtained from the isomorphism $\xi^{-1}$ of Definition
\ref{def1} followed by the identification \eqref{e4}.
Define the map
\[
\overline{\varphi}:\otimes^{m}V_{i}\to
\otimes^{m}(V^*_{m}\otimes V_{i})=
  (\otimes^{m}V^*_{m})\otimes (\otimes^{m}V_{i}), \,\,
  \overline{\varphi}(v)  = v_{o}\otimes v.
\]
Clearly, the map $\overline{\varphi}$ is a $\GL(V_{i})$-module map.
Moreover, it restricts to the map $\overline{\varphi}_{1}$:
\[
\xymatrix{
S^{m}(V_{i})\ar[r]^-{\overline{\varphi}_{1}}\ar[d] &
S^{m}(V^*_{m}\otimes V_{i})\ar[d]\\
\otimes^{m}(V_{i})\ar[r]^-{\overline{\varphi}} & \otimes^{m}(V^*_{m}\otimes
V_{i}),
}
\]
where the vertical maps are the canonical inclusions.

It is easy to see that $\overline{\varphi_1}$ is a $\GL(V_{i})$-module map
with image inside $S^{m}(V^*_{m}\otimes
V_{i})^{-(\epsilon_{1}+\cdots+\epsilon_{m}), \mathfrak{S}_{m}}$.
Thus, from the irreducibility of $S^{m}V_{i}$ as $\GL(V_i)$-module, applying
the Schur's lemma, we can choose the identification  \eqref{e7} so that
$\overline{\varphi}_{1}$ coincides
with the map $\varphi_{|S^{m}(V_{i})}$ under the identification  \eqref{e8}.

The map $\overline{\varphi}_{1}:S^{m}(V_{i})\to
S^{m}(V^*_{m}\otimes V_{i})$ extends to an algebra homomorphism (still denoted by)
$$\overline{\varphi}_{1}:S^\bullet(S^{m}(V_{i}))\to
S^\bullet(S^{m}(V^*_{m}\otimes V_{i})).$$
 The isomorphism \eqref{e8} for $j=1$:
\begin{equation}\label{e9}
\mathcal{P}^{m}(M(m,i)) \simeq
 S^{m}(V^{*}_{m}\otimes V_{i})
\end{equation}
induces an algebra homomorphism $\beta: S^\bullet(S^{m}(V^*_{m}\otimes V_{i}))
\to \mathcal{P}^{m\bullet}(M(m,i)).$ Let $\overline{\varphi}'_{1}:
S^\bullet(S^{m}(V_{i}))\to \mathcal{P}^{m\bullet}(M(m,i))$ be the $\GL(V_i)$-algebra
homomorphism which is the composite $\beta\circ \overline{\varphi}_{1}$.

Since $\overline{\varphi}'_{1}$ coincides with $\varphi$ on
$S^{m}(V_{i})$, and both $\varphi$ and $\overline{\varphi}'_{1}$ are
algebra homomorphisms, we get that
\beqn\label{e+}
\overline{\varphi}'_{1}=\varphi.
\eeqn

Consider the function (for $i\leq m$)
$$
\theta: M(m,i) \to \mathcal{P}^m(E_i)\simeq S^{m}(E_{i})^*,\,\,\,
A \mapsto (\mathscr{D}\odot A)_{|{E_{i}}}.
$$
Explicitly,
$$
\theta(A)\left(\sum^{i}_{j=1}\lambda_{j}e_{j}\right)=\prod^{m}_{p=1}
\left(\sum^{i}_{j=1}\lambda_{j}a_{p}^j\right),\quad\text{for}\quad
A= (a_{p}^j)_{1\leq p\leq m, 1\leq j\leq i}.
$$
Clearly, $\theta$ is a polynomial function of homogeneous degree $m$. Moreover, it
is $\GL(E_{i})$-equivariant:
\begin{align*}
\theta(A\cdot g^{-1}) &= (\mathscr{D}\odot (Ag^{-1}))_{|{E_{i}}}\\
&= g\cdot ((\mathscr{D}\odot A)_{|{E_{i}}})\\
&= g\cdot\theta(A).
\end{align*}

Of course, $\theta$ gives rise to a $\GL(E_{i})$-algebra
homomorphism
$$
\theta^{*}: {S}^{\bullet}(S^{m}(E_{i}))\to
{\mathcal{P}}^{m \bullet }(M(m,i)).
 $$

\begin{lemma}\label{lemma3.1}
$\text{\rm Im\,}\left(\theta^{*}_{|{S^{m}(E_{i})}}\right)\subset
  \mathcal{P}^{m}(M(m,i))^{-(\epsilon_1+\cdots+\epsilon_{m}), \mathfrak{S}_m}$.
\end{lemma}

\begin{proof}
Let $\bf{t}$ be the diagonal matrix $(t_{1}, \dots, t_m)\in \GL(m).$
  For any $f\in S^{m}(E_{i})$,
\begin{align*}
(\theta^{*}f)({\bf t}^{-1}A) &= f\left((\mathscr{D}\odot
  {\bf t}^{-1}A)_{|{E_{i}}}\right)\\
&= t^{-1}_{1}\dots t^{-1}_{m}f\left((\mathscr{D}\odot A)_{|{E_{i}}}\right).
\end{align*}
This shows that
$$
\text{Im\,}\left(\theta^{*}_{|S^{m}(E_{i})}\right)\subset
\mathcal{P}^{m}(M(m,i))^{-(\epsilon_{1}+\cdots+\epsilon_{m})}.
$$
We next show that for any $f\in S^{m}(E_{i})$, $\theta^{*}f$ is
$\mathfrak{S}_{m}$-invariant: Take $\sigma\in \mathfrak{S}_{m}$ (considered as a
permutation matrix), then
\begin{align*}
(\theta^{*}f)(\sigma A) &= f\left((\mathscr{D}\odot \sigma
  A)_{|E_{i}}\right)\\
&= f((\mathscr{D}\odot A)_{|E_i}).
\end{align*}
This proves the lemma.
\end{proof}

Since the function $\theta$ is clearly nonzero, we see that
$\theta^{*}$ coincides (up to a nonzero scalar multiple in any degree) with the function
$\varphi:{S}^{\bullet}(S^{m}(E_{i}))\to \mathcal{P}^{m\bullet}
(M(m,i))^{-\bullet(\epsilon_{1}+\cdots+\epsilon_{m}), \mathfrak{S}_m}$ defined
earlier.
Now, $S^{i}(S^{m}(E_{i}))$ has a unique (up to a scalar multiple) $\SL(E_{i})$-invariant (by
Proposition \ref{prop2.3}). We want to determine if
$\theta^{*}_{|[S^{i}(S^{m}(E_{i}))]^{\SL(E_{i})}}\neq 0$?

By the definition, $S^{j}(S^{m}(V_{i}))=
[\otimes^{j}(\otimes^{m}V_{i})]^{H_j},$ where $H_j\subset \mathfrak{S}_{mj}$ is
the subgroup
$\mathfrak{S}^{\times j}_{m}\rtimes \mathfrak{S}_{j}$ acting as
\begin{align*}
&
  (({\sigma}_{1},\dots,\sigma_{j}),\mu)\cdot\left((v^{1}_{1}\otimes\cdots\otimes
  v^{1}_{m})\otimes\cdots\otimes (v^{j}_{1}\otimes\cdots\otimes
  v^{j}_{m})\right)\\
&= \left(v^{\mu(1)}_{\sigma_{1}(1)}\otimes\cdots\otimes
  v^{\mu(1)}_{\sigma_{1}(m)}\right)\otimes\cdots\otimes
  \left(v^{\mu(j)}_{\sigma_{j}(1)}\otimes\cdots\otimes
  v^{\mu(j)}_{\sigma_{j}(m)}\right),
\end{align*}
for $\sigma_p\in \mathfrak{S}_{m}$ and $\mu \in \mathfrak{S}_{j}$.
\begin{proposition}\label{prop8}
The map $\varphi_{|[S^{i}(S^{m}(E_{i}))]^{\SL(E_{i})}}\neq 0$ if and only if
the $\GL(V_{m})$-submodule $U_i$ generated by $v_{o}^{\otimes i}\in
S^{i}(S^{m}(V^{*}_{m}))= [\otimes^{i}(\otimes^{m}V^{*}_{m})]^{H_i}$ intersects
the  isotypic
component $\mathcal{I}_{m\delta_i}$ of
$S^i(S^{m}(V^{*}_{m}))$ corresponding to the irreducible $\GL(V_{m})$-module
$V_m(m\delta_i)^*$ nontrivially.
\end{proposition}

\begin{proof}
Take $0\neq v\in
[S^{i}(S^{m}(E_{i}))]^{\SL(E_{i})}=[\otimes^{i}(\otimes^{m}E_{i})]^{H_i\times \SL(E_i)}$.

Recall that for any partition $\lambda: \lambda_1 \geq \dots\geq \lambda_d > 0$, $d$
is called the {\it height} $ht\, \lambda$ of $\lambda$. We set
$|\lambda|:=\sum \lambda_j$. Let $W_\lambda$ be the corresponding
irreducible $\mathfrak{S}_{|\lambda|}$-module  and let $V_i(\lambda)$ be the
corresponding irreducible
$\GL(i)$-module for any $i\geq d$. By the Schur-Weyl duality (cf.
[GW, Thoerem 9.1.2]),
$$
S^{i}(S^{m}(E_{i}))\simeq \bigoplus_{\substack{ht\,\lambda\leq
    i\\ |\lambda|=mi}}[W_{\lambda}]^{H_i}\otimes V_{i}(\lambda).
$$
Thus, we get
\begin{equation}\label{e10}
[S^{i}(S^{m}(E_{i}))]^{\SL(E_{i})}\simeq [W_{m\delta_{i}}]^{H_i}
\otimes V_i(m\delta_i).
\end{equation}
In particular,
 $[W_{m\delta_{i}}]^{H_i}$ is one dimensional.
Also, consider the analogous decomposition:
\begin{equation}\label{e11}
S^{i}(S^{m}(V^{*}_{m}))\simeq \bigoplus_{\substack{ht\, \mu\leq
    m\\ |\mu|=mi}}[W_{\mu}]^{H_i}\otimes V_{m}(\mu)^{*},
\end{equation}
and write
\begin{equation}\label{e11}
v_{o}^{\otimes i}=\sum_{\mu}v_{\mu},
\end{equation}
 under the above decomposition.

Let $M\subset [\otimes^{i}(\otimes^{m}E_{i})]^{\SL(E_i)}$ be the $\mathfrak{S}_{mi}$-submodule
generated by $v$ and, for any $\mu$ with $ht\,\mu \leq
    m$ and $|\mu|=mi$, let $M_\mu \subset \otimes^{i}(\otimes^{m} V_m^*)$ be the
$\mathfrak{S}_{mi}$-submodule
generated by $v_\mu$. Then, by the Schur-Weyl duality, $M\simeq W_{m\delta_i}$ and
$M_\mu$ is isotypic of type $W_\mu$. By the definition,
\begin{align}\label{e12}
\bar{\varphi}'_1 (v)&= \frac{1}{im!}\sum_{\sigma\in \mathfrak{S}_{mi}}\sigma\cdot
(v_{o}^{\otimes i}\otimes v)\in S^{mi}(V_m^*\otimes E_i)\subset
(\otimes^i(\otimes^m V_m^*))\otimes (\otimes^i(\otimes^m E_i))\notag\\
&= \frac{1}{im!}\sum_{\substack{ht\,\mu\leq
    m\, ,|\mu|=mi}}\,\sum_{\sigma\in \mathfrak{S}_{mi}}\sigma\cdot (v_\mu \otimes v).
\end{align}
Now, $W_\mu$ being self dual, we get that for $\mu \neq m\delta_i$,
\begin{equation}\label{e13}
\sum_{\sigma\in \mathfrak{S}_{mi}}\sigma\cdot (v_\mu \otimes v) =0.
\end{equation}
Moreover, if $v_{m\delta_i}\neq 0$, we claim that
\begin{equation}\label{e14}
\sum_{\sigma\in \mathfrak{S}_{mi}}\sigma\cdot (v_{m\delta_i}\otimes v) \neq 0:
\end{equation}
By projecting to an ireducible component, we can assume that
$M_{m\delta_i} \simeq W_{m\delta_i}.$ Now, take a $\mathfrak{S}_{mi}$-invariant
nondegenerate bilinear form $\alpha: M_{m\delta_i} \otimes M\to \mathbb{C}$. Since
$\alpha$ is $\mathfrak{S}_{mi}$-invariant, $\alpha_{|M_{m\delta_i}^{H_i} \otimes
M^{H_i}}$ remains nondegenerate. Since both of
$v$ and $v_{m\delta_i}$ are $H_i$-invariant, and $[W_{m\delta_i}]^{H_i}$ is one
dimensional, we get $\alpha (v_{m\delta_i} \otimes v)\neq 0.$ Thus,
\begin{align*}
\alpha\bigl(\sum_{\sigma \in \mathfrak{S}_{mi}}\,\sigma\cdot (v_{m\delta_i}
\otimes v)\bigr)&= \sum_{\sigma \in \mathfrak{S}_{mi}}\alpha (v_{m\delta_i} \otimes v)\\
&\neq 0.
\end{align*}
This proves \eqref{e14}. Now, as it is easy to see, $v_{m\delta_i} \neq 0$ if and only if $U_i$ intersects
$\mathcal{I}_{m\delta_i}$ nontrivially. Hence the prposition is proved by
combining the identities \eqref{e12}- \eqref{e14} since $\varphi=\bar{\varphi}'_1
$ (by the identity \eqref{e+}).
\end{proof}

\section{Latin  Squares}

\begin{definition}
Let $1\leq i\leq m$. By a {\em Latin $(i,m)$-rectangle} $A$, one means a
 $i\times m$ matrix
$$
A=\left(a_{p}^q\right)_{\substack{1\leq p\leq i\\ 1\leq q\leq m}}
$$
such that each row $A_{p}:=\{a_{p}^1,\,\ldots,\,a_{p}^m\}$ is a permutation
 $\sigma_{p}$ of $[m]$ (i.e., $\sigma_{p}(q)=a_{p}^q)$
 and each column $A^{q}:=\{a_{1}^q,\,\ldots,\,a_{i}^q\}$ consists of distinct
 numbers. We define the {\it sign} $\epsilon (A^{q})$ of $A^{q}$ as follows:
$$
\epsilon(A^{q}):=\text{~ sign of~}\prod_{1\leq p<p'\leq i}\left(a_{p'}^q-a_{p}^q\right).
$$
Call a Latin rectangle $A$ {\em column-even} if $\epsilon_{c}(A):=\prod\limits^{m}_{q=1}\epsilon(A^{q})$ is $+1$ and {\em column-odd} otherwise.
\end{definition}

Let $\mathcal{A}^{q}$  denote the set $A^{q}$ without regard to the
order. Then, we call the $m$-tuple
$\mathcal{A}=\left(\mathcal{A}^{1},\,\ldots,\,\mathcal{A}^{m}\right)$ the
{\em pattern} of $A$.
Let $L_{\mathcal{A}}$ denote the set of Latin $(i,m)$-rectangles $A$
with pattern $\mathcal{A}$. Let $S(i,m)$ be the set of all patterns of
size $(i,m)$, where by a {\it pattern $\mathcal{A}$ of size $(i,m)$} (or an
$(i,m)$-{\it pattern}) we mean a $m$-tuple
$\mathcal{A}=(\mathcal{A}^1,\,\ldots,\,\mathcal{A}^{m})$ of subsets of $[m]$, each of
cardinality exactly $i$ such that any integer $q\in[m]$ occurs in
exactly $i$-many sets $\mathcal{A}^\bullet$.

For $\mathcal{A}\in S(i,m)$, let $L^+_{\mathcal{A}}$
(resp. $L^{-}_{\mathcal{A}}$) denote the subset of $L_{\mathcal{A}}$
consisting of column even (resp. odd) Latin rectangles.

We have the following simple lemma.

\begin{lemma}\label{lem15}
Fix any $1\leq i\leq m$. Assume that there exists a pattern
$\mathcal{A}$ of size $(i,m)$ such that
$$
\sharp L^{+}_{\mathcal{A}}\neq \sharp L^{-}_{\mathcal{A}}.
$$

Then, for any $1\leq i'\leq i$, there exists a pattern $\mathcal{B}$ of
size $(i',m)$ such that
$$
\sharp L^{+}_{\mathcal{B}}\neq \sharp L^{-}_{\mathcal{B}}.
$$
\end{lemma}

\begin{proof}
It suffices to prove the lemma for $i'=i-1$.
Define the map $\varphi:L_{\mathcal{A}}\to
\bigsqcup\limits_{\mathcal{B}\in S(i-1,m)}L_{\mathcal{B}}$, by removing
the last row of any Latin rectangle $A$ in $L_{\mathcal{A}}$.
The map $\varphi$ is clearly injective. Moreover, the image of
$\varphi$ consists exactly of the union $\bigsqcup\limits_{\mathcal{B}\in
    S_{\mathcal{A}}(i-1,m)}\,L_{\mathcal{B}},$ where
\[
 S_{\mathcal{A}}(i-1,m):=
\left\{\begin{array}{p{8cm}}
(i-1,m){\rm -patterns $\mathcal{B}$ such that there exists $A\in
  L_{\mathcal{A}}$ with its top $(i-1)$-rows having pattern $\mathcal{B}$}
\end{array}\right\}.
\]
By the definition of $L^{\pm}_{\mathcal{A}}$, it is clear that for any
$\mathcal{B}\in S_{\mathcal{A}}(i-1,m)$, there exists a sign
$\epsilon(\mathcal{B})\in \{\pm 1\}$ such that

\beqn \label{e15} \varphi^{-1}\left(L^{\pm}_{\mathcal{B}}\right)\subset L^{\pm
  \epsilon(\mathcal{B})}_{\mathcal{A}}.
  \eeqn

Assume, if possible, that the lemma is false, i.e.,

\beqn \label{e161} \sharp L^{+}_{\mathcal{B}}=\sharp L^{-}_{\mathcal{B}}\,,\,\,\,
\text{for every $(i-1,m)$-pattern}\,\, \mathcal{B};
\eeqn
 in particular, for any
$\mathcal{B}\in S_{\mathcal{A}}(i-1,m)$.

Combining \eqref{e15} and \eqref{e161}, we get (since $\varphi$ is a bijection) $\sharp
L^{+}_{\mathcal{A}}=\sharp L^{-}_{\mathcal{A}}$.
This contradicts the assumption and hence proves the lemma.
\end{proof}

We recall the following celebrated {\em column Latin $(m,m)$-square conjecture}
 due to Huang-Rota [HR, Conjecture 3].

\begin{conjecture}\label{conj16}
{\it For any positive even integer $m$,
$$\sharp CELS(m)\neq \sharp COLS(m),$$
where $CELS(m)$ (resp. $COLS(m)$) denotes the set of
 column-even Latin
$(m,m)$-squares (resp. column-odd Latin
 $(m,m)$-squares).

(Observe that for Latin $(m,m)$-squares, there is a unique pattern:
$([m],\,[m],\,\ldots,\,[m])$.)}
\end{conjecture}

Combining the above conjecture with Lemma \ref{lem15}, we get the
following proposition.

\begin{proposition}\label{prop17}
Let $m$ be a positive even integer. Assume that the above column Latin
$(m,m)$-square conjecture \ref{conj16} is true.

Then, for any $1\leq i\leq m$, there exists a pattern $\mathcal{A}$ of size
$(i,m)$ such that
$$
\sharp L^{+}_{\mathcal{A}}\neq \sharp L^{-}_{\mathcal{A}}.
$$
\end{proposition}
\begin{remark}\label{rem18}
As proved by Huang-Rota [HR, \S 3], their column Latin
$(m,m)$-square conjecture is equivalent to the (full) Latin
$(m,m)$-square conjecture given by Alon-Tarsi [AT]. Now, the
(full) Latin $(m,m)$-square conjecture is valid in the following
cases:
\begin{itemize}
\item[(a)] $m=p-1$, for any odd prime $p$, due to Glynn [G, Theorem
  3.2],

\item[(b)] $m=p+1$, for any odd prime $p$, due to Drisko [D].
\end{itemize}
\end{remark}

We have the following very simple lemma.

\begin{lemma}\label{lem19}
Let $\mathcal{A}$ (resp.  $\mathcal{B}$) be a pattern of type $(i,m)$
(resp.  $(i,m')$) such that
$$
\sharp L^{+}_{\mathcal{A}}\neq \sharp
L^{-}_{\mathcal{A}}\quad\text{and}\quad \sharp L^{+}_{\mathcal{B}}\neq \sharp
L^{-}_{\mathcal{B}}.
$$

Then,
$$
\sharp L^{+}_{(\mathcal{A},\mathcal{B})}\neq \sharp
L^{-}_{(\mathcal{A},\mathcal{B})},
$$
where each entry in $\mathcal{B}$ is shifted by $m$.
\end{lemma}

\begin{proof}
Clearly, under the concatenation
$$
L_{\mathcal{A}}\times L_{\mathcal{B}}\xrightarrow{\sim}
L_{(\mathcal{A},\mathcal{B})}.
$$
Moreover, under the above bijection,
$$
L^{\epsilon_{1}}_{\mathcal{A}}\times L^{\epsilon_{2}}_{\mathcal{B}}\to
 L^{\epsilon_{1}\cdot\epsilon_{2}}_{(\mathcal{A},\mathcal{B})},
$$
where $\epsilon_{i}=\pm 1$.
From this the lemma follows.
\end{proof}

\section{Existence of a certain isotypic component in the module generated by
$v_o^{\otimes i}$}

Recall that $V_m=\mathbb{C}^m$ has standard basis $\{v_1, \dots, v_m\}$.
Recall from the identity \eqref{e100},
$$
\mathfrak{v}_{o}:=\dfrac{1}{m!}\sum_{\sigma_{1}\in \mathfrak{S}_{m}}v_{\sigma_1(1)}\otimes\cdots\otimes
v_{\sigma_1(m)}\in S^{m}(V_{m}),
$$
so that, as elements of $S^{i}(S^{m}(V_{m}))$,
$$
\mathfrak{v}^{\otimes
  i}_{o}=\dfrac{1}{(m!)^{i}}\sum_{\sigma=(\sigma_{1},\,\ldots,\,\sigma_{i})\in
  \mathfrak{S}^{i}_{m}}\left(v_{\sigma_1(1)}\otimes\cdots\otimes
v_{\sigma_{1}(m)}\right)\otimes\cdots\otimes
\left(v_{\sigma_{i}(1)}\otimes\cdots\otimes v_{\sigma_{i}(m)}\right).
$$

Let $\lambda$ be a partition of $k$ into at most $m$ parts and let $A$ be a tableau
of shape $\lambda$. As in [GW, Proposition 9.3.7], define
\begin{align}\label{e17}
\XRow A &= \{\sigma\in \mathfrak{S}_{k}:\sigma\text{~ preserves the
  rows of $A$}\},\notag\\
\XCol A &= \{\mu\in \mathfrak{S}_{k}:\mu\text{~ preserves the
  columns of $A$}\},\notag\\
S(A) &= \left(\sum_{\mu\in \XCol A}\epsilon (\mu)\mu\right)\cdot
\sum_{\sigma\in\XRow A}\sigma \,, \,\,\,\text{and}\notag\\
v_{A}&=v_{i_{1}}\otimes\cdots\otimes v_{i_{k}}\in \otimes^k(V_m),
\end{align}
where $i_{j}=r$
if $j$ occurs in the $r$-th row of $A$. (Here $\epsilon (\mu)$ denotes the
sign of $\mu$.)

\begin{exam*}
\begin{gather*}
A=
\begin{array}{|c|c|c|}
\hline
1 & 5 & 9\\
\hline
2 & 6 & 10\\
\hline
3 & 7 & \multicolumn{1}{c}{}\\
\cline{1-2}
4 & 8 & \multicolumn{1}{c}{}\\
\cline{1-2}
\end{array}\\
v_{A}=v_{1}\otimes v_{2}\otimes v_{3}\otimes v_{4}\otimes v_{1}\otimes
v_{2}\otimes v_{3}\otimes v_{4}\otimes v_{1}\otimes v_{2}.
\end{gather*}
\end{exam*}

Consider the tableau $B_{o}=B_{o}(i,m)$ of shape $m\geq m\geq
\cdots\geq m$ ($i$-factors):

\begin{center}
\begin{tabular}{|c|c|c|c|c|}
\hline
1 & 2 & 3 & $\ldots$ & $m$\\
\hline
$m+1$ & $m+2$ & $m+3$ & $\ldots$ & $2m$\\
\hline
$\vdots$ & $\vdots$ & $\vdots$ &$\vdots$ & $\vdots$\\
\hline
$(i-1)m+1$ & $(i-1)m+2$ & $(i-1)m+3$ & $\ldots$ & $im$\\
\hline
\end{tabular}
\end{center}
\vskip2ex

\begin{proposition}\label{prop20}
For any $1\leq i\leq m$ and even $m$,
$$
\left\langle v^{\otimes i}_{o},\,S(B_{o})\cdot\mathfrak{v}^{\otimes
  i}_{o}\right\rangle =\left(\frac{1}{m!}\right)^i\sum_{\mathcal{A}\in S(i,m)}(\sharp
L^{+}_{A}-\sharp L^{-}_{A})^{2},
$$
where $\langle\,,\,\rangle$ is the standard pairing between
$
\otimes^{i}(\otimes^{m}(V^{*}_{m}))$,
$\otimes^{i}(\otimes^{m}V_{m})$ and $v_o\in S^{m}(V^{*}_{m}) \subset
\otimes^{m}(V^{*}_{m})$ is defined by the identity \eqref{e6}.
\end{proposition}

\begin{proof}
First of all, by the definition of $S(B_{o})$,
\begin{align} \label{e13}
S(B_{o})\cdot \mathfrak{v}_{o}^{\otimes i}&=(m!)^{i}
\sum_{\mu\in \XCol \,B_o}\epsilon (\mu)\,\mu \cdot \mathfrak{v}^{\otimes i}_{o}\notag\\
&= \sum_{\substack{\sigma=(\sigma_{1},\,\ldots,\,\sigma_{i})\in
      \mathfrak{S}^{i}_{m}\\
      \mu=(\mu_{1},\,\ldots,\,\mu_{m})\in
      \mathfrak{S}^{m}_{i}}}\epsilon (\mu) \left(v_{\sigma_{\mu_{1}(1)}(1)}\otimes\cdots\otimes
  v_{\sigma_{\mu_{m}(1)}(m)}\right)\otimes \notag\\
&\quad \left(v_{\sigma_{\mu_{1}(2)}(1)}\otimes\cdots\otimes
  v_{\sigma_{\mu_{m}(2)}(m)}\right)\otimes \cdots \otimes
  \left(v_{\sigma_{\mu_{1}(i)}(1)}\otimes\cdots\otimes
  v_{\sigma_{\mu_{m}(i)}(m)}
  \right),
\end{align}
where $\epsilon(\mu):=\epsilon(\mu_{1})\ldots\epsilon(\mu_{m})$ and
$\mu_{j}$ is embedded in $\mathfrak{S}_{mi}$ as permuting
$\{j,\,j+m,\,\ldots,\,j+(i-1)m\}$ only.

For any $i\times m$ matrix $A=(a_{p,q})_{\substack{1\leq p\leq
    i\\ 1\leq q\leq m}}$ of integers $a_{p,q}\in[m]$, let
$$
V_{A}:=  \left(v_{a_{1,1}}\otimes v_{a_{1,2}}\otimes\cdots\otimes
v_{a_{1,m}}\right)\otimes\cdots\otimes
 \left(v_{a_{i,1}}\otimes v_{a_{i,2}}\otimes\cdots\otimes
v_{a_{i,m}}\right)\in \otimes^{i}(\otimes^{m}V_{m}).
$$
Thus, we can rewrite the identity \eqref{e13} as
$$
S(B_{o})\cdot \mathfrak{v}_{o}^{\otimes
  i}=\sum_{\substack{\sigma\in \mathfrak{S}^{i}_{m}\\ \mu\in
    \mathfrak{S}^{m}_{i}}}\epsilon (\mu)V_{A(\sigma,\mu)}\,,
$$
where $A(\sigma,\mu)$ is the $i\times m$ matrix
$$
A(\sigma,\mu)=
\begin{pmatrix}
\sigma_{\mu_{1}(1)}(1) & \ldots & \sigma_{\mu_{m}(1)}(m)\\
\vdots & & \vdots\\
\sigma_{\mu_{1}(i)}(1) & \ldots & \sigma_{\mu_{m}(i)}(m)
\end{pmatrix}.
$$

We claim that
\begin{equation}\label{e18}
\left\langle v^{\otimes i}_{o},\, S(B_{o})\cdot \mathfrak{v}_o^{\otimes
  i}\right\rangle = \left\langle v^{\otimes
  i}_{o},\,\sum\limits_{(\sigma,\mu)\in R}\epsilon (\mu)
V_{A(\sigma,\mu)}\right\rangle,
\end{equation}
where the last summation runs over $R$ consisting of those
$\sigma=(\sigma_{1},\,\ldots,\,\sigma_{i})\in \mathfrak{S}^{i}_{m}$ and
$\mu=(\mu_{1},\,\ldots,\,\mu_{m})\in \mathfrak{S}^{m}_{i}$ such that
$A(\sigma,\mu)$ is a Latin $(i,m)$-rectangle:

Since $v_{o}$ is, by
definition, $\sum\limits_{\sigma_{1}\in
  \mathfrak{S}_{m}}v^{*}_{\sigma_1(1)}\otimes\cdots\otimes
v^{*}_{\sigma_{1}(m)}$, unless each row of $A(\sigma,\mu)$ is a
permutation of $[m]$, we have $\langle v^{\otimes
  i}_{o},\,V_{A(\sigma,\mu)}\rangle=0$. Further, assume that the entries in
some column of $A(\sigma,\mu)$ are non distinct, say
$$
\sigma_{\mu_{q}(p)}(q)=\sigma_{\mu_{q}(p')}(q),\quad\text{for
  some}\quad 1\leq q\leq m\,\,\,\text{and some}\,\, 1\leq p\neq p'\leq i.
$$
 Let $\tau\in \mathfrak{S}_{i}$ be the
transposition $(p,p')$. Then,
$$
V_{A(\sigma,\mu)}=V_{A(\sigma,\mu')},
$$
where $\mu':=(\mu_{1},\,\ldots,\,\mu_{q}\circ \tau,\,\ldots,\,\mu_{m})$.

Hence,
$$
\epsilon(\mu)V_{A(\sigma,\mu)}+\epsilon(\mu')V_{A(\sigma,\mu')}=0.
$$

This proves the identity \eqref{e18}.

Let $R'\subset \mathfrak{S}^{i}_{m}$ be the subset consisting of
$\sigma=(\sigma_{1},\,\ldots,\,\sigma_{i})$ such that
$$
A(\sigma)=\begin{pmatrix}
\sigma_{1}(1) & \ldots & \sigma_{1}(m)\\
\vdots & & \vdots\\
\sigma_{i}(1) & \ldots & \sigma_{i}(m)
\end{pmatrix}
$$
is a Latin $(i,m)$-rectangle. For any $\sigma\in R'$, let
$\widehat{\sigma}$ be the pattern
$(\widehat{\sigma}^{1},\,\ldots,\,\widehat{\sigma}^{m})$, where
$$
\widehat{\sigma}^{q}:=\{\sigma_{1}(q),\,\ldots,\,\sigma_{i}(q)\}.
$$

Define an equivalence relation on $R'$ by $\sigma\sim \sigma'$ if the
patterns $\widehat{\sigma}=\widehat{\sigma'}$.
Denote the equivalence class containing $\sigma$ by $[\sigma]$.
Then, the sum $\sum\limits_{(\sigma,\mu)\in R}\epsilon
(\mu)V_{A(\sigma,\mu)}$ can clearly be written  as
\begin{align*}
\sum_{\sigma\in R'}\sum_{\substack{\mu\in
    \mathfrak{S}^{m}_{i}:\\ (\sigma,\mu)\in R}}\epsilon
(\mu)V_{A(\sigma,\mu)}&=\sum\limits_{\sigma\in
  R'}\epsilon_{c}(A(\sigma ))\sum_{B\in
  L_{\widehat{\sigma}}}\epsilon_{c}(B)V_{B}\\
& =\sum_{[\sigma]\in R'/\sim}\,\,\sum_{A\in
  L_{\widehat{\sigma}}}\epsilon_{c}(A)\sum_{B\in
  L_{\widehat{\sigma}}}\epsilon_{c}(B)V_{B}\\
 &= \sum_{\mathcal{A}\in S(i,m)}\sum_{A\in
    L_{\mathcal{A}}}\epsilon_{c}(A)\sum_{B\in
    L_{\mathcal{A}}}\epsilon_{c}(B)V_{B}.
\end{align*}

Thus, by the identify \eqref{e18},
\begin{align*}
\langle v^{\otimes i}_{o},\,S(B_{o})\cdot \mathfrak{v}^{\otimes i}_{o}\rangle
&= \left(\frac{1}{m!}\right)^i\,\sum_{\mathcal{A}\in S(i,m)}\,\sum_{A\in
  L_{\mathcal{A}}}\epsilon_{c}(A)\sum_{B\in
  L_{\mathcal{A}}}\epsilon_{c}(B)\\
&= \left(\frac{1}{m!}\right)^i \,\sum_{\mathcal{A}\in S(i,m)}\,\left(\sum_{A\in
  L_{\mathcal{A}}}\epsilon_{c}(A)\right)^{2}\\
&= \left(\frac{1}{m!}\right)^i\,\sum_{\mathcal{A}\in S(i,m)}\left(\sharp L^{+}_{\mathcal{A}}- \sharp
L^{-}_{\mathcal{A}}\right)^{2}.
\end{align*}
This proves the proposition.
\end{proof}

As an immediate consequence of the above proposition, we get:

\begin{corollary}\label{coro21}
$\langle v^{\otimes i}_{o},\,S(B_{o})\cdot \mathfrak{v}^{\otimes
    i}_{o}\rangle\neq0$ if and only if for some pattern
  $\mathcal{A}\in S(i,m)$, $\sharp L^{+}_{\mathcal{A}}\neq \sharp
  L^{-}_{\mathcal{A}}$.
\end{corollary}

For any partition $\lambda$ of $k$ into at most $m$ parts, let
$G^{\lambda}$ be the highest weight space in $\otimes^{k}(V_{m})$ for
$GL(V_{m})$ corresponding to the highest weight $\lambda$.
Then, we have the following lemma (cf. [GW, Lemma
    9.3.2]).

\begin{lemma} \label{lemma9.3.2} Let $A$ be a tableau of shape $\lambda$. Then,
$S(A)\cdot v_{A}$ is a nonzero element of $G^{\lambda}$. Thus,
$$G^{\lambda} =\sum_{\tau\in \mathfrak{S}_k}\,\mathbb{C} \tau\cdot (S(A)\cdot v_{A}).
$$
\end{lemma}

We specialize the above lemma to $k=m^2$ and $\lambda$ the partition:
$$
m\delta_m:\,\,\,\,\,\,\,\underbrace{m\geq m\geq\cdots\geq m}_{m\text{-factors}}.
$$
In this case $V_m(m\delta_m)$ is a one dimensional representation of
$GL(V_{m})$.

Consider the tableau $B_o=B_o(m,m)$ (with $i=m$) given just above
Proposition \ref{prop20}. Then,
\begin{align}\label{e21}
S(B_{o})\cdot
v_{B_{o}}=&(m!)^m \sum_{\mu=(\mu_{1},\,\ldots,\,\mu_{m})\in \mathfrak{S}^m_{i=m}}
\epsilon (\mu) \left(v_{\mu_{1}(1)}\otimes
v_{\mu_{2}(1)}\otimes \cdots\otimes
v_{\mu_{m}(1)}\right)\otimes\notag\\
& \left(v_{\mu_{1}(2)}\otimes v_{\mu_{2}(2)}\otimes
\cdots\otimes v_{\mu_{m}(2)}\right)\otimes
 \cdots\otimes \left(v_{\mu_{1}(m)}\otimes
v_{\mu_{2}(m)}\otimes\cdots\otimes v_{\mu_{m}(m)}\right).
\end{align}

By the above lemma, the isotypic component $G^{\lambda}$ of
$\otimes^m (\otimes^m\,V_m)$ for the partition $\lambda = m\delta_m$ is
the span
of
$$
\left\{\tau\cdot (S(B_{o})\cdot v_{B_{o}}):\tau\in \mathfrak{S}_{m^2}\right\}.
$$
I thank J. Landsberg for the (b)-part of the following proposition.
\begin{proposition}\label{lem13}
(a)  $\langle v^{\otimes m}_{o},\,S(B_{o})\cdot
  v_{B_{o}}\rangle =$
  $$\sharp \,CELS(m) - \sharp \, COLS(m),$$
  where $CELS$ and $COLS$ are defined in Conjecture \ref{conj16}.

 (b)
For any $\tau\in \mathfrak{S}_{m^{2}}$,
 $$\langle v^{\otimes m}_{o},\,\tau\cdot (S(B_{o})\cdot
  v_{B_{o}})\rangle=\alpha\langle v^{\otimes m}_{o},\,S(B_{o})\cdot
  v_{B_{o}}\rangle,\,\,\,\text{for some}\,\,\alpha \in \{0,\pm 1\}.
  $$
\end{proposition}
\begin{proof}
By the identity \eqref{e21},
$$
\langle v^{\otimes m}_{o},\,S(B_{o})\cdot v_{B_{o}}\rangle
=\sum \epsilon ({\bf \mu}),
$$
where the summation runs over those
${\bf {\mu}}=(\mu_{1},\,\ldots,\,\mu_{m})\in \mathfrak{S}_{m}^m$ such
that
$$
B(\mu):=
\begin{pmatrix}
\mu_{1}(1) & \mu_{2}(1) & \cdots & \mu_{m}(1)\\
\mu_{1}(2) & \mu_{2}(2) & \cdots & \mu_{m}(2)\\
\vdots & \vdots & & \vdots\\
\mu_{1}(m) & \mu_{2}(m) & \cdots& \mu_{m}(m)
\end{pmatrix}
$$
is a Latin square (i.e., each row and each column of the above matrix
is a permutation of $[m]$), and
$$
\epsilon ({\bf \mu}):=\epsilon (\mu_{1})\cdots \epsilon (\mu_{m}).
$$
From this the (a)-part follows.

By the expression of $S(B_{o})\cdot v_{B_{o}}$ as in the identity \eqref{e21}, clearly
\begin{align*}
\tau\cdot (S(B_{o})\cdot
v_{B_{o}})  =(m!)^m\,\sum_{{\bf \mu}=(\mu_{1},\,\ldots,\,\mu_{m})\in
  \mathfrak{S}_{m}^m}& \epsilon
({\bf \mu})\left(v_{\mu_{i_{1}^1}(j^1_{1})}\otimes
\cdots\otimes v_{\mu_{i_{m}^1}(j^1_{m})}\right)\otimes\\
&\quad \cdots\otimes
\left(v_{\mu_{i^{m}_{1}}(j^{m}_{1})}\otimes\cdots\otimes
v_{\mu_{i^{m}_{m}}(j^{m}_{m}})\right),
\end{align*}
for some fixed $i^{p}_{q}$ and $j^{p}_{q}\in [m]$.

We claim that if for any $1\leq p\leq m$, $i^{p}_{a}=i^{p}_{b}=:q$ for
some $a\neq b$, then $D_{\tau}=0$, where $D_{\tau}:=\langle v^{\otimes
  m}_{o},\,\tau\cdot (S(B_{o})\cdot v_{B_{o}})\rangle$. Observe that
  $j^{p}_{a} \neq j^{p}_{b}$ since the element
  $$\left(v_{\mu_{i_{1}^1}(j^1_{1})}\otimes
\cdots\otimes v_{\mu_{i_{m}^1}(j^1_{m})}\right)\otimes
\cdots\otimes
\left(v_{\mu_{i^{m}_{1}}(j^{m}_{1})}\otimes\cdots\otimes
v_{\mu_{i^{m}_{m}}(j^{m}_{m}})\right)$$  is a permutation of
 $$\left(v_{\mu_{1}(1)}\otimes
v_{\mu_{2}(1)}\otimes \cdots\otimes
v_{\mu_{m}(1)}\right)\otimes
\left(v_{\mu_{1}(2)}\otimes v_{\mu_{2}(2)}\otimes
\cdots\otimes v_{\mu_{m}(2)}\right)\otimes
 \cdots\otimes \left(v_{\mu_{1}(m)}\otimes
v_{\mu_{2}(m)}\otimes\cdots\otimes v_{\mu_{m}(m)}\right).$$ Consider the
element $\theta=(j^{p}_{a},\,j^{p}_{b})\in \mathfrak{S}_{m}$. Then, replacing
$\mu_{q}$ by $\mu_{q}\theta$ in the above expression of $\tau\cdot
(S(B_{o})\cdot v_{B_{o}})$, we clearly get
$$
D_{\tau}=\epsilon (\theta) D_{\tau}.
$$
Thus, $D_{\tau}=0$.

So, let us assume that for any $1\leq p\leq m$,
$
i^{p}_{a}\neq i^{p}_{b}\quad\text{for}\quad a\neq b.
$
Since $v^{\otimes m}_{o}$ is $H_{m}$-invariant (where $H_m$ is defined above
Proposition \ref{prop8}), to
calculate $D_{\tau}$, we can assume that
$$
\tau\cdot (S(B_{o})\cdot v_{B_{o}})=(m!)^m\,\sum_{{\bf \mu}\in
  \mathfrak{S}^{m}_{m}}\epsilon
({\bf \mu})\left(v_{\mu_{1}(j^1_{1})}\otimes \cdots\otimes
v_{\mu_{m}(j^1_{m})}\right)\otimes \cdots \otimes
\left(v_{\mu_{1}(j^{m}_{1})}\otimes \cdots\otimes
v_{\mu_{m}(j^{m}_{m})}\right),
$$
where, for any $1\leq q\leq m$, $\{j^1_{q},\,\ldots,\,j^{m}_{q}\}$ is a
permutation $\sigma_{q}$ of $[m]$.
Now, replacing $\mu_{q}$ by $\mu_{q}\circ \sigma_{q}$, we get
(setting ${\bf \sigma}=(\sigma_{1},\,\ldots,\,\sigma_{m})$)
\begin{align*}
\tau\cdot (S(B_{o})\cdot v_{B_{o}}) &= \epsilon
({\bf \sigma})(m!)^m\,\sum_{{\bf \sigma}\in \mathfrak{S}^{m}_{m}}\epsilon
({\bf \mu})\left(v_{\mu_{1}(1)}\otimes\cdots\otimes
v_{\mu_{m}(1)}\right)\otimes \cdots \\
&\otimes \left(v_{\mu_{1}(m)}\otimes
\cdots\otimes v_{\mu_{m}(m)}\right)\\
&= \epsilon ({\bf \sigma})\,S(B_{o})\cdot v_{B_{o}}.
\end{align*}
This proves the proposition.
\end{proof}
\begin{theorem} \label{thm23} Let $m$ be an even positive integer and let
 $1\leq i\leq m$. If there exists a pattern $\mathcal{B}$ of size
 $(i,m)$ such that
\beqn \label{e16}
\sharp L^{+}_{\mathcal{B}}\neq \sharp L^{-}_{\mathcal{B}},
\eeqn
then the $\GL(V_m)$-submodule $U_i$ generated by $v^{\otimes i}_{o}
\in S^i(S^m(V_m^*))= [\otimes^{i}(\otimes^{m}V^{*}_{m})]^{H_i}$
intersects the isotypic component $\mathcal{I}_{m\delta_{i}}$ of
$S^i(S^m(V_m^*))$ corresponding to the irreducible  $\GL(V_m)$-module
$V_m(m\delta_i)^*$ nontrivially, where $H_i$ is defined over Proposition
\ref{prop8}.

In particular, if the column Latin $(m,m)$-square conjecture \ref{conj16} is true, then
$U_i\cap \mathcal{I}_{m\delta_{i}} \neq (0)$, for all $1\leq i\leq m$.

For $i=m$, $U_m\cap \mathcal{I}_{m\delta_{m}} \neq (0)$ if and only if
the column Latin $(m,m)$-square conjecture  is true.
\end{theorem}
\begin{proof} Let $y_o =v_{m\delta_{i}}$ be the component of
$v^{\otimes i}_{o}$ in  $\mathcal{I}_{m\delta_{i}}$  (cf. the identity \eqref{e11}).
Then, as observed in the proof of Proposition \ref{prop8}, $U_i\cap \mathcal{I}_{m\delta_{i}}\neq 0$ if and only
 if $y_o\neq 0$. By [GW, Theorem 9.3.10], $S(B_{o})\cdot\mathfrak{v}^{\otimes
  i}_{o}$ belongs to an irreducible $\GL(V_m)$-submodule of
  $\otimes^{i}(\otimes^{m}V_{m})$ of highest weight $m\delta_i$. Thus,
  \begin{align*}
  \langle y_o,\,S(B_{o})\cdot\mathfrak{v}^{\otimes
  i}_{o}\rangle &= \langle v_o^{\otimes i},\,S(B_{o})\cdot\mathfrak{v}^{\otimes
  i}_{o}\rangle\\
  &=\left(\frac{1}{m!}\right)^{i}\sum_{\mathcal{A}\in S(i,m)}(\sharp
L^{+}_{A}-\sharp L^{-}_{A})^{2}, \,\, \text{by Proposition \ref{prop20}}\\
&\neq 0\,\,\,\text{by the assumption of the theorem.}
\end{align*}
Thus, $y_o\neq 0$, proving the first part of the theorem.

The second `In particular' part of the thoerem, of course, follows from Lemma
\ref{lem15}.

For the last part, by Lemma \ref{lemma9.3.2},
$S(B_o)\cdot v_{B_o}$ is a nonzero highest weight vector of $\otimes^{m}
(\otimes^{m}V_{m})$ with highest weight $m\delta_{m}$ and the isotypic component of $\otimes^{m}
(\otimes^{m}V_{m})$ corresponding to the  highest weight $m\delta_{m}$ is given by
$\sum_{\tau \in \mathfrak{S}_{m^2}}\,\mathbb{C}\,\tau\cdot
\bigl(S(B_o)\cdot v_{B_o}\bigr)$ (since $V_{m}(\delta_m)$ is a one dimensional
representation).

Thus, $y_o\in  \mathcal{I}_{m\delta_{m}}$ is nonzero if and only if
$$\langle v^{\otimes m}_{o},\,x\rangle=\langle y_o,\,x\rangle\neq 0,$$
for some $x\in \sum_{\tau \in \mathfrak{S}_{m^2}}\,\mathbb{C}\,\tau\cdot
\bigl(S(B_o)\cdot v_{B_o}\bigr)$. The above condition is equivalent to the
nonvanishing of
$\langle v^{\otimes m}_{o},\, S(B_o)\cdot v_{B_o}\rangle$ by Proposition
\ref{lem13} (b); which, in turn, is equivalent to the validity of the
 column Latin $(m,m)$-square conjecture (Conjecture \ref{conj16})  by
 Proposition \ref{lem13} (a).
 This proves the theorem.
 \end{proof}
 \begin{remark} It is quite possible that for any $1\leq i\leq m$,
 $U_i\cap  \mathcal{I}_{m\delta_{i}} \neq 0$ if and only if
 the equation \eqref{e16} is satisfied for some pattern
 $\mathcal{B}$ of size $(i,m)$.
 \end{remark}

\section{Statement of the main theorem and its consequences}\label{sec3}

Let $\mathfrak{v}$ be a complex vector space of dimension $m$ and let
$E:=\mathfrak{v}\otimes \mathfrak{v}^{*}=\End \mathfrak{v}$, $Q:=\mathcal{P}^m (E)
\simeq S^{m}(E)^{*}$ (under the isomorphism $\xi$ of Definition \ref{def1}). Consider
$\mathscr{D}\in Q$, where $\mathscr{D}$ is the function taking
determinant of any $A\in E=\End \mathfrak{v}$. The group $G=GL(E)$ acts
canonically on $Q$. Let $\mathcal{X}$ be the $G$-orbit closure of
$\mathscr{D}$ inside $Q$.

Fix a basis $\{v_{1}, \ldots, v_{m}\}$ of $\mathfrak{v}$ and let
$\{v^{*}_{1}, \ldots, v^{*}_{m}\}$ be the dual basis of $\mathfrak{v}^{*}$. Take
the basis $\{v_{i}\otimes v^{*}_{j}\}_{1\leq i,j\leq m}$ of $E$ and
order them as $\{v_{1},v_{2},\ldots,v_{m^{2}}\}$ satisfying
$$
v_{1}=v_{1}\otimes v^{*}_{1},\quad v_{2}=v_{2}\otimes
v^{*}_{2},\ldots,v_{m}=v_{m}\otimes v^{*}_{m}.
$$

{\em Assume that $m$ is even}. Recall from Corollary \ref{coro2.4}
that for any $1\leq i\leq m^{2}$, the irreducible $GL(E)$-module
$V_{E}(m\delta_{i})$ occurs in $S^{i}(S^{m}(E))$ with multiplicity
one (and $V_{E}(m\delta_{i})$ does not occur in any
$S^{j}(S^{m}(E))$, for $j\neq  i$). Let $P_{i}=\gamma_{m,i}\in
S^{i}(S^{m}(E))$ be the highest weight vector of
$V_{E}(m\delta_{i})\subset S^{i}(S^{m}(E))$ (which is unique up to a
nonzero scalar multiple) with respect to the standard Borel subgroup
$B=B_{E}$ of $G$ consisting of upper triangular invertible matrices,
where $GL(E)$ is identified with $GL(m^{2})$ with respect to the
basis $\{v_{1},\ldots,v_{m^{2}}\}$ of $E$ given above. By Lemma
\ref{lem2.1}, in fact $P_i\in [S^{i}(S^{m}(E_i))]^{\SL(E_i)}$, where
$E_i$ is the subspace of $E$ spanned by $\{e_1, \dots, e_i\}$.

Recall an explicit construction of $P_{i}$ from Lemma \ref{lem2.6}.
Since $P_{i}\in S^{i}(S^{m}(E))$, we can
think of $P_{i}$ as a homogeneous polynomial of degree $i$ on the
vector space $Q=S^{m}(E)^{*}$.

The following is our main result.

\begin{theorem}\label{thm3.1}
Assume, as above, that $m$ is even. Assume further that the column Latin
$(m,m)$-square conjecture \ref{conj16} is true. Then, with the above notation, for
any $1\leq i\leq m$, the polynomial $P_{i}$ does {\em not} vanish
identically on $\mathcal{X}$.

In particular, the irreducible $GL(E)$-module $V_{E}(m\delta_{i})$
occurs with multiplicity one in the affine coordinate ring
$\mathbb{C}[\mathcal{X}]$. Moreover, by Corollary \ref{coro2.4},
$V_{E}(d\delta_{i})$, for any $d<m$ and any $1\leq i\leq m^{2}$, does
not occur in $S^{\bigdot}(S^{m}(E))$; in particular, it does not occur
in $\mathbb{C}[\mathcal{X}]$.
\end{theorem}
\begin{proof} Recall the definition of the right action of the semigroup
End $E$ on $Q=\mathcal{P}^m(E)$ from the identity \eqref{e-1}. Consider the map
$$\hat{\theta}:M(m,i) \to Q, \,\,\, A\mapsto \mathscr{D}\odot \hat{A},$$
where $\hat{A} \in \End \,E$  is defined by
$$\hat{A} e_j= \sum_{p=1}^m a^j_pe_p,\,\,\,\text{for}\,\, 1\leq j \leq i,
\,\,\text{and}\,\, \hat{A} e_j=0, \,\,\,\text{for}\,\, j> i,$$
where $A=(a^j_p)_{1\leq p\leq m, 1\leq j\leq i}.$ Clearly,
$$\text{Im}\, \hat{\theta} \subset \mathcal{X}.$$
To prove that $P_i\in \mathcal{P}^i(Q)\simeq S^{i}(S^{m}(E))$
restricts to a nonzero function on $\mathcal{X}$, it suffices to show that
$P_i$ restricts to a nonzero function on $M(m,i)$ via the morphism
$\hat{\theta}$. Since
$$P_i\in S^{i}(S^{m}(E_i))\simeq \mathcal{P}^i(\mathcal{P}^m(E_i)),$$
$P_i$ is the pull-back of a function $\bar{P}_i\in
\mathcal{P}^i(\mathcal{P}^m(E_i))$ via the restriction map $r:
\mathcal{P}^m(E)\to \mathcal{P}^m(E_i)$. Thus, it suffices to prove that
$\bar{P}_i$ restricts to a nonzero function on $M(m,i)$ via
$\theta: M(m,i) \to \mathcal{P}^m(E_i)\simeq S^m(E_i)^*$ defined as the composite
$\theta= r\circ \hat{\theta}$. (Observe that this $\theta$ coincides with the
 map $\theta$ defined just before Lemma \ref{lemma3.1}.) Now, as observed just
 before Proposition \ref{prop8}, the induced map
 $$\theta^*:  S^\bullet(S^m(E_i))\to \mathcal{P}^{m\bullet}(M(m,i))$$
 coincides with the map $\varphi$. Since $\bar{P}_i$ is the unique
 (up to a nonzero multiple) nonzero element of $[S^{i}(S^{m}(E_i))]^{\SL(E_i)}$,
 it suffices to show that $\varphi_{|[S^{i}(S^{m}(E_i))]^{\SL(E_i)}} \neq 0$. This
 follows from Proposition \ref{prop8} and Theorem \ref{thm23}, and hence the
theorem is proved.
\end{proof}

\begin{corollary}\label{coro3.2}
With the notation and assumptions as in the last theorem (in particular,
assuming the validity of the column Latin
$(m,m)$-square conjecture), for any
dominant integral weight $\lambda$ for $GL(E)$ of the form
$\lambda=\sum^{m}_{i=1} n_i\delta_{i}$, $n_{i}\in \mathbb{Z}_{+}$, the
irreducible $GL(E)$-module $V_{E}(m\lambda)$ occurs in
$\mathbb{C}[\mathcal{X}]$ with nonzero multiplicity.
\end{corollary}

\begin{proof}
First of all, $\mathcal{X}$ being an irreducible variety,
$\mathbb{C}[\mathcal{X}]$ is an integral domain. Take a
$B_{E}$-eigenvector $\tilde{P}_{i}\in \mathbb{C}[\mathcal{X}]$ of
weight $m\delta_{i}$ for any $1\leq i\leq m$; which exists by the last
theorem (assuming the validity of the column Latin
$(m,m)$-square conjecture). Now, consider the function
$$
\tilde{P}_{\lambda}=\prod^{m}_{i=1}\tilde{P}^{n_{i}}_{i}\in
\mathbb{C}[\mathcal{X}].
$$

Clearly, $\tilde{P}_{\lambda}$ is a nonzero $B_{E}$-eigenvector of
weight $m\lambda$. This proves the Corollary.
\end{proof}

Let $\mathcal{X}^{o}$ be the $G$-orbit $G\cdot \mathscr{D}\subset
Q$. Then, by a classical result due to Frobenius
(cf. \cite[Proposition 2.1 and Corollary 2.3]{K}), the
isotropy subgroup $G_{\mathscr{D}}$ of $\mathscr{D}$ is a reductive
subgroup. In particular, by a result of Matsushima, $\mathcal{X}^{o}$
is an affine variety. Moreover, by Frobenius reciprocity, we get the
following.

\begin{proposition}\label{prop3.3}
$\mathbb{C}[\mathcal{X}^{o}]\simeq
\bigoplus_{\lambda}V_{E}(\lambda)\otimes
  [V_{E}(\lambda)^{*}]^{G_{\mathscr{D}}}$ as $G$-modules, where
  the  summation runs over all the dominant integral weights
  $\lambda$ of $G$ (i.e., $\lambda$ runs over $\sum^{m^{2}}_{i=1}
  n_i\delta_{i}$, $n_{i}\in \mathbb{Z}_{+}$ for all $1\leq i<m^{2}$ and
  $n_{m^{2}}\in \mathbb{Z}$) and
  $[V_{E}(\lambda)^{*}]^{G_{\mathscr{D}}}$ denotes the subspace of
  $G_{\mathscr{D}}$-invariants in the dual space
  $V_{E}(\lambda)^{*}$. The action of $G$ on the right side is via its
  standard action on the first factor and it acts trivially on the
  second factor.

In particular, the multiplicity of $V_{E}(\lambda)$ in
$\mathbb{C}[\mathcal{X}^{o}]$ is the dimension of the invariant space
$[V_{E}(\lambda)^{*}]^{G_{\mathscr{D}}}$.
\end{proposition}

Considering the action of the centre of $G$, it is easy to see that if
$V_{E}(\lambda)$ occurs in $\mathbb{C}[\mathcal{X}^{o}]$, then
$|\lambda|:=\sum^{m^{2}}_{i=1}i
\,n_i\in m\mathbb{Z}$, where (as earlier)
$\lambda=\sum^{m^{2}}_{i=1}n_{i}\delta_{i}$.

Applying [BLMW, Proposition 5.2.1], we get that for any polynomial representation
$V_{E}(\lambda)$ (i.e., $\lambda=\sum^{m^{2}}_{i=1}n_{i}\delta_{i}$
with all $n_{i}\in \mathbb{Z}_{+}$) with $|\lambda|=md$, $d\in
\mathbb{Z}_{+}$,
\begin{equation}\label{e65}
\dim \left[V_{E}(\lambda)^{*}\right]^{G_{\mathscr{D}}}=
sk_{\overline{\lambda}, d{\delta}_{m},d{\delta}_{m}},
\end{equation}
where ${\delta}_{m}$ (as earlier) is the partition
${\delta}_{m}: (1\geq 1\geq \cdots \geq 1)$ ($m$ factors),
$\overline{\lambda}$ is the partition $(n_{1}+\cdots+n_{m^{2}}\geq
n_{2}+\cdots+n_{m^{2}}\geq n_{3}+\cdots+n_{m^{2}}\geq\cdots \geq
n_{m^{2}}\geq 0)$ and
$sk_{\overline{\lambda}, d{\delta}_{m},d{\delta}_{m}}$
is the symmetric Kronecker coefficient (i.e., the multiplicity of the irreducible
$\mathfrak{S}_{dm}$-module $W_{\overline{\lambda}}$ in the symmetric  product
$S^2(W_{d{\delta}_{m}})$, where, as earlier,
$W_{\overline{\lambda}}$ denotes the irreducible $\mathfrak{S}_{dm}$-module
corresponding to the partition $\overline{\lambda}$).

As a corollary of the equation \eqref{e65}, and Proposition \ref{prop3.3}, we
get the following (since $\mathbb{C}[\mathcal{X}]\hookrightarrow
\mathbb{C}[\mathcal{X}^{o}]$ is a $G$-module).

\begin{corollary}\label{coro3.4}
For any irreducible polynomial representation $V_{E}(\lambda)$ of $G$,
such that $|\lambda|=dm$, for $d\in \mathbb{Z}_{+}$, the multiplicity
$\mu(\lambda)$ of $V_{E}(\lambda)$ in $\mathbb{C}[\mathcal{X}]$ is
bounded by:
$$
\mu(\lambda)\leq sk_{\overline{\lambda}, d{\delta}_{m},
d{\delta}_{m}}.
$$
\end{corollary}

Observe that unless $V_{E}(\lambda)$ is a polynomial representation of
$G$ and $|\lambda|\in m\mathbb{Z}_{+}$, we have $\mu(\lambda)=0$.

As an immediate consequence of Corollaries \ref{coro3.2} and
\ref{coro3.4}, we get the following.

\begin{corollary}\label{coro3.5}
Let $m$ be any positive even integer. Assume that the column Latin
$(m,m)$-square conjecture is true. Then, for any partition
$\overline{\lambda}: \left(\overline{\lambda}_{1}\geq
\overline{\lambda}_{2}\geq \cdots\geq \overline{\lambda}_{m}\geq
0\right) $ (with at most $m$ parts) of $d$
(i.e., $|\overline{\lambda}|=d)$, the symmetric Kronecker coefficient
$$
sk_{m\overline{\lambda}, d{\delta}_{m},d{\delta}_{m}} > 0.
$$
\end{corollary}

\begin{remark}\label{rem5.4}
(a) Compare the above corollary with \cite[Theorem 1, \S\ 3]{BCI}.

(b) The
following  generalization of Theorem \ref{thm3.1} holds by exactly
the same proof.

   Let $\mathscr{F} \in Q=S^m(E^*)$ be any (homogeneous) polynomial such that
   writing $\mathscr{F}$ as a sum of monomials (in a basis of $E^*$), some
   monomial with no repeated factors occurs with nonzero
   coefficient. Assume further that the column Latin
$(m,m)$-square conjecture \ref{conj16} is true. Then, for even
 $m$ and any $1\leq i\leq m$,
   the polynomial
  $P_{i}$ does not vanish identically on the orbit $GL(E)\cdot
  \mathscr{F}$.

  In particular, this remark applies to $\mathscr{F}=\mathfrak{P}$,
  where $\mathfrak{P}$ is the function $E\to \mathbb{C}$
  taking any matrix $A\in E:=\End \mathfrak{v}$ to its permanent.

Thus, the irreducible $GL(E)$-module $V_{E}(m\delta_{i})$
occurs with multiplicity one in
$\mathbb{C}[\overline{GL(E)\cdot\mathfrak{P}}]$ for any $1\leq i\leq
m$ (assuming the validity of the column Latin
$(m,m)$-square conjecture). Moreover, $V_{E}(d\delta_{i})$, for any $d<m$ and $1\leq i\leq
m^{2}$ does not occur in $\mathbb{C}[\overline{GL(E)\cdot \mathfrak{P}}]$
(cf. Corollary \ref{coro2.4}).
\end{remark}

\vskip3ex
\noindent
S.K.: Department of Mathematics, University of North Carolina,
Chapel Hill, NC 27599-3250, USA (email: shrawan$@$email.unc.edu)

\end{document}